\documentclass{amsart}
\usepackage[foot]{amsaddr}

\usepackage[left=1in,right=1in,top=1in,bottom=1in]{geometry}
\usepackage{tikz}
\usetikzlibrary{shapes,snakes,calc,arrows}
\usetikzlibrary{decorations.pathreplacing,shapes.geometric}
\usetikzlibrary{calc,positioning}

\usepackage{amsrefs}
\usepackage{amsthm}
\usepackage{amscd}
\usepackage{amsfonts}
\usepackage{amsmath}
\usepackage{amssymb}
\usepackage{mathrsfs}
\usepackage{enumitem}
\usepackage{multirow}
\usepackage{verbatim}
\usepackage{url}
\usepackage{graphicx}
\usepackage{color}

\usepackage{ifthen}

\usepackage{microtype}
\usepackage{hyperref}

\usepackage{wasysym} 

\usepackage{breqn}

\usepackage[T1]{fontenc} 

\usepackage{tipa}
\newcommand{\blob}{vaccine}
\newcommand{\fpf}{\paren{(\A_\ell^{(\iota)}, \A_r^{(\iota)})}_{\iota\in\I}}

\definecolor{blue1}{rgb}{0.3 0.2 0.9}
\definecolor{green1}{rgb}{0.2 0.8 0.4}
\definecolor{blue2}{rgb}{0.2 0.8 1}
\definecolor{green2}{rgb}{0.0 0.6 0.0}
\definecolor{dgrey}{rgb}{0.1 0.1 0.1}
\definecolor{lgrey}{rgb}{0.85 0.85 0.85}

\def\greyscale{1}
\ifthenelse{\greyscale>0}{
	\def\palette{{"dgrey", "lgrey"}}
	\definecolor{rline}{rgb}{0 0 0}
}{
	\def\palette{{"blue2", "green2"}}
	\definecolor{rline}{rgb}{1 0 0}
}

\definecolor{ggreen}{HTML}{00BB33}

\begin{document}


\newcommand{\supp}{\text{supp}}
\newcommand{\Aut}{\text{Aut}}
\newcommand{\Gal}{\text{Gal}}
\newcommand{\Inn}{\text{Inn}}
\newcommand{\Irr}{\text{Irr}}
\newcommand{\Ker}{\text{Ker}}
\newcommand{\N}{\mathbb{N}}
\newcommand{\Z}{\mathbb{Z}}
\newcommand{\Q}{\mathbb{Q}}
\newcommand{\R}{\mathbb{R}}
\newcommand{\C}{\mathbb{C}}
\renewcommand{\H}{\mathcal{H}}
\newcommand{\B}{\mathcal{B}}
\newcommand{\A}{\mathcal{A}}
\newcommand{\Y}{\mathcal{Y}}
\newcommand{\X}{\mathcal{X}}
\newcommand{\I}{\mathcal{I}}
\newcommand{\K}{\mathcal{K}}
\newcommand{\M}{\mathcal{M}}
\newcommand{\E}{\mathcal{E}}
\newcommand{\e}{\epsilon}
\newcommand{\GG}{\Gamma\Gamma}
\newcommand{\GGab}{\GG\paren{(a_1, b_1), \ldots, (a_n, b_n); m, k, \e}}
\newcommand{\GGo}{\underline{\Gamma\Gamma}}
\newcommand{\GGoab}{\GGo\paren{(a_1, b_1), \ldots, (a_n, b_n); m, k, \e}}

\newcommand{\J}{\mathscr{J}}
\newcommand{\D}{\mathscr{D}}

\renewcommand{\a}{\alpha}

\newcommand{\ul}[1]{\underline{#1}}

\newcommand{\II}{\text{II}}
\newcommand{\III}{\text{III}}

\newcommand{\<}{\left\langle}
\renewcommand{\>}{\right\rangle}
\renewcommand{\Re}[1]{\text{Re}\ #1}
\renewcommand{\Im}[1]{\text{Im}\ #1}
\newcommand{\dom}[1]{\text{dom}\,#1}
\renewcommand{\i}{\text{i}}
\newcommand{\alg}{\operatorname{alg}}
\newcommand{\mb}[1]{\mathbb{#1}}
\newcommand{\mc}[1]{\mathcal{#1}}
\newcommand{\mf}[1]{\mathfrak{#1}}
\newcommand{\mr}{\mathrm}
\newcommand{\im}{\operatorname{im}}

\newcommand{\lat}{\mathrm{lat}}
\newcommand{\lrleq}{\leq_{\mathrm{lat}}}
\newcommand{\lrgeq}{\geq_{\mathrm{lat}}}

\newcommand{\paren}[1]{\left(#1\right)}
\newcommand{\ang}[1]{\left<#1\right>}
\newcommand{\set}[1]{\left\{#1\right\}}
\newcommand{\sq}[1]{\left[#1\right]}
\newcommand{\abs}[1]{\left|#1\right|}
\newcommand{\norm}[1]{\left\Vert#1\right\Vert}

\newcommand{\makeaball}[1]{\begin{tikzpicture}\pgfmathparse{\palette[#1]}\shade[ball color=\pgfmathresult] (0, 0) circle (0.07);\end{tikzpicture}}

\newcommand{\taur}{\text{\taurus}}


\newtheorem{thm}{Theorem}
\newtheorem{prop}[thm]{Proposition}
\newtheorem{lem}[thm]{Lemma}
\newtheorem{cor}[thm]{Corollary}
\newtheorem{innercthm}{Theorem}
\newenvironment{cthm}[1]
  {\renewcommand\theinnercthm{#1}\innercthm}
  {\endinnercthm}
\newtheorem{innerclem}{Lemma}
\newenvironment{clem}[1]
  {\renewcommand\theinnerclem{#1}\innerclem}
  {\endinnerclem}

\theoremstyle{definition}
\newtheorem{defn}[thm]{Definition}
\newtheorem{ex}[thm]{Example}
\newtheorem{nota}[thm]{Notation}
\newtheorem{exam}[thm]{Example}
\newtheorem{rem}[thm]{Remark}
\newtheorem{innercdefi}{Definition}
\newenvironment{cdefi}[1]
  {\renewcommand\theinnercdefi{#1}\innercdefi}
  {\endinnercdefi}
\newtheorem{cons}[thm]{Construction}


\author{Ian~Charlesworth}
\address{University of California, Berkeley\\Department of Mathematics\\Berkeley, CA, USA, 94720-3840}
\email{ilc@math.berkeley.edu}
\thanks{Research supported by NSF grants DMS-1161411 and DMS-1500035, and NSERC award PGS-6799-438440-2013.} 


\title{An alternating moment condition for bi-freeness.}

\begin{abstract}
	In this note we demonstrate an equivalent condition for bi-freeness, inspired by the well-known ``vanishing of alternating centred moments'' condition from free probability.
	We show that all products satisfying a centred condition on maximal monochromatic $\chi$-intervals have vanishing moments if and only if the family of pairs of faces they come from is bi-free, and show that similar characterisations hold for the amalgamated and conditional settings.
	In addition, we construct a bi-free unitary Brownian motion and show that conjugation by this process asymptotically creates bi-freeness; these considerations lead to another characterisation of bi-free independence.
\end{abstract}

\maketitle

\section{Introduction.}
Bi-free probability was introduced by Voiculescu in \cite{voiculescu2014free} as a generalisation of free probability studying simultaneously left and right actions of algebras on a reduced free product space.
Voiculescu demonstrated that many notions from free probability generalise with appropriate care to this bi-free setting.
In particular, \cite{voiculescu2014free} demonstrated the existence of bi-free cumulant polynomials but did not produce explicit formulae for them.
Soon after, Mastnak and Nica in \cite{mastnak2015double} proposed a family of cumulant functionals, which the author together with Nelson and Skoufranis in \cite{2014arXiv1403.4907C} showed to agree with those abstractly given by Voiculescu.
In particular, \cite{2014arXiv1403.4907C} demonstrated that bi-freeness was equivalent to the vanishing of mixed cumulants.

Since then, many more techniques from free probability have been generalised to the bi-free setting: bi-free partial transforms were studied in \cite{voiculescu2016free, voiculescu2016freeiii, skoufranis2015combinatorial, skoufranis2016independences, huang2016analytic}; infinite divisibility and a bi-free L\'evy-Hin\v{c}in formula in \cite{gu2015analogue}; bi-matrix models in \cite{skoufranis2015some}; and so on.
One major difficulty in generalising results to bi-free probability, however, has been that the condition defining bi-freeness is somewhat unwieldy.
Free independence is equivalent to saying that alternating products of centred elements are themselves centred, but in bi-free probability it has been necessary to work either with cumulants (and hence compute the M\"obius function for the lattice of bi-non-crossing partitions) or to compute moments through abstract bi-free products.
In this note, we demonstrate an appropriate bi-free analogue of the freeness condition.
We also examine the extension of these techniques to conditional bi-free probability as studied by Gu and Skoufranis \cite{gu2016conditionally}, and to certain operator-valued bi-free settings as considered in \cite{Charlesworth:2015aa,voiculescu2014free}.

Moreover, we consider a bi-free analogue of Biane's free multiplicative Brownian motion, which is a free stochastic process that may be constructed as the solution to a free stochastic equation involving free (additive) Brownian motion, or as a limit of the Markov process arising from the heat semi-group on finite dimensional unitary matrices \cite{biane1997free}.
A free multiplicative Brownian motion $U(t)$ converges in moments to a Haar unitary operator as $t\to\infty$; we introduced bi-free multiplicative Brownian motion, a pair of stochastic processes which behave similarly and converge in moments to a bi-Haar unitary in the sense of \cite{Charlesworth:2015aa}.
Conjugation by a Haar unitary moves algebras into free position, and conjugating pairs of faces by bi-Haar unitaries moves them into bi-free position: thus conjugating by (bi-)free multiplicative Brownian motion asymptotically creates (bi-)free independence as $t\to\infty$, and so may be thought of as a ``liberation process'' (\emph{cf.} \cite{Voiculescu1999101} for the free case).
We examine the effect of conjugation by bi-free multiplicative Brownian motion for small $t$, and show that the derivative at time $0$ may be computed through a combinatorially-described map $\taur$, which bears relation to Voiculescu's free liberation gradient \cite{Voiculescu1999101}.
Incidentally, this gives us another characterization of bi-free independence in terms of the vanishing of the state composed with this gradient.

In addition to this introduction, this note contains three sections.
Section~\ref{sec:prelim} recalls some preliminaries of bi-free probability.
Section~\ref{sec:vaccine} defines the \emph{vaccine} property, and shows that it is equivalent to bi-freeness.
Finally, Section~\ref{sec:liberation} defines bi-free unitary Brownian motion, and uses it to define the liberation process.

\subsection*{Acknowledgements.}
The author would like to thank Brent Nelson and Dimitri Shlyakhtenko for useful discussions and remarks on earlier drafts of this note.
Significant progress on this paper was made during the ``von Neumann Algebras'' Hausdorff Trimester Program and the author is grateful for the hospitality and support of the Hausdorff Research Institute for Mathematics in Bonn, Germany.

\section{Preliminaries.}
\label{sec:prelim}
A \emph{non-commutative probability space} is a pair $(\A, \varphi)$ with $\A$ a unital $*$-algebra, and $\varphi$ a state on $\A$.
A \emph{pair of faces} in $\A$ is a pair of unital sub-algebras $(\A_\ell, \A_r)$ of $\A$.
A family of pairs of faces $(\A_\ell^{(\iota)}, \A_r^{(\iota)})_{\iota \in \I}$ is said to be \emph{bi-freely independent} if they can be represented using left and right actions on a free product of vector spaces such that the corresponding joint distribution matches their joint distribution in $\A$; see \cite{voiculescu2014free} for more details.
It was shown in \cite{voiculescu2014free} that bi-free independence uniquely determines all joint moments of the family of pairs of faces in terms of pure moments consisting of terms coming from single pairs of faces.

We now introduce some notation and combinatorial objects used in \cite{2014arXiv1403.4907C}, and the reader wishing a more careful treatment is encouraged to look there.
Given a map $\chi : \set{1, \ldots, n} \to \set{\ell, r}$, we enumerate $\chi^{-1}(\ell) = \set{i_1 < \cdots < i_p}, \chi^{-1}(r) = \set{i_{p+1} > \cdots > i_n}$, and define the permutation $s_\chi \in \mathcal{S}_n$ by $s_\chi(k) = i_k$.
We will write $i \prec_\chi j$ if $s_\chi^{-1}(i) < s_\chi^{-1}(j)$, and call a set $I \subseteq \set{1, \ldots, n} $ a \emph{$\chi$-interval} if it is an interval under the ordering $\prec_\chi$, or equivalently if the set $s_\chi^{-1}\cdot I$ is an interval in the usual sense.
The set of \emph{bi-non-crossing partitions corresponding to $\chi$} is then
$$BNC(\chi) := \set{\pi \in \mathcal{P}(n) : s_\chi^{-1}\cdot\pi \in NC(n)} = \set{\pi \in \mathcal{P}(n) : \substack{i\prec_\chi j \prec_\chi k \prec_\chi \ell\\\text{ and } i \sim_\pi k, j \sim_\pi \ell} \Rightarrow i \sim_\pi j}.$$
If nodes are placed on two adjacent lines according to $\chi$ and labelled appropriately, the bi-non-crossing partitions are precisely those which can be drawn without crossings.
We order the set of partitions $\mc{P}(n)$ (and hence $BNC(\chi)$) by refinement: $\pi < \sigma$ if and only if every block in $\sigma$ is a union of blocks in $\pi$.

Although technically the $BNC(\chi)$ are not disjoint (e.g., $\set{\set{1, \ldots, n}} \in BNC(\chi)$ for all $\chi$), it will be very convenient to treat them as formally different and allow $\pi \in BNC(\chi)$ to ``remember'' $\chi$.
To accomplish this, one could take the elements of $BNC(\chi)$ to be pairs $(\pi, \chi)$ and project onto the first coordinate at every turn.
The extra notation required, though, mainly serves to occlude the arguments and so we confine the subtlety to this paragraph.

Given a map $\e : \set{1, \ldots, n} \to \mathcal{I}$, we say a set $I$ is \emph{$\e$-monochromatic} if $\e|_I$ is constant.
Note that $\e$ induces a partition in $\mc{P}(n)$ by $\set{\e^{-1}(\iota) : \iota \in I} \setminus\set\emptyset$, and so we will sometimes write $\pi < \e$ for $\pi \in BNC(\chi)$.
This corresponds precisely to saying that every block in $\pi$ is $\e$-monochromatic.

\begin{exam}
	\label{ex:partition}
	Suppose $\chi$ and $\e$ are such that $\chi^{-1}(\ell) = \set{2,3,4,7}$, $\chi^{-1}(r) = \set{1,5,6,8}$, $\e^{-1}\paren{\makeaball{0}} = \set{1,2,4,7,8}$, and $\e^{-1}\paren{\makeaball{1}} = \set{3,5,6}$.
	
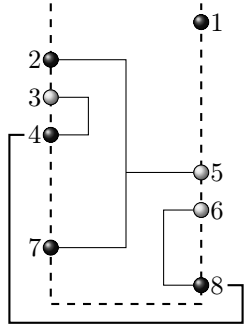
\[\begin{tikzpicture}[baseline]
		\draw[thick, dashed] (-1,0.25) -- (-1, -3.75) -- (1,-3.75) -- (1,0.25);

		\def\colours{{0,0,1,0,1,1,0,0}}
		\def\sidez{{1,-1,-1,-1,1,1,-1,1}}
		\foreach \y in {0,...,7} {
			\pgfmathtruncatemacro{\nodename}{\y+1}
			\pgfmathtruncatemacro{\sd}{\sidez[\y]}
			\pgfmathparse{\palette[\colours[\y]]}
			\def\clr{\pgfmathresult}
			\node (ball\nodename) [draw, shade, circle, ball color=\clr, inner sep=0.07cm] at (\sd, -\y*0.5) {};
			\ifthenelse{\sd=1}{\node[right] at (\sd, -\y*0.5) {\nodename}}
					{\node[left] at (\sd, -\y*0.5) {\nodename}};
		}
		\draw[rline,thick] (ball4)++(-0.35,0) -- ++(-0.2,0) |- (0,-4);
		\draw[rline,thick] (ball8)++(0.35,0) -- ++(0.2,0) |- (0,-4);

		\draw (ball2) [right] -- (ball2 -| 0,0 ) |- (ball7) [right];
		\draw (ball5) [left] -- (ball5 -| 0,0 );
		\draw (ball6) [left] -- (ball6 -| 0.5,0 ) |- (ball8) [left];
		\draw (ball3) [right] -- (ball3 -| -0.5,0 ) |- (ball4) [right];
	\end{tikzpicture}\]
Then $\set{\set{1}, \set{2,5,7}, \set{3,4}, \set{6,8}} \in BNC(\chi)$ but is not a non-crossing partition in the usual sense.
Also, $\set{4,7,8}$ is a maximal $\e$-monochromatic $\chi$-interval.
We have $2\prec_\chi 3 \prec_\chi 4\prec_\chi 7\prec_\chi8\prec_\chi 6\prec_\chi5\prec_\chi1$, so $\prec_\chi$ corresponds to reading the order of the nodes around the diagram.
\end{exam}

The cumulants are multilinear functionals $\kappa_\chi : \A^n \to \C$ defined recursively by the moment-cumulant relation:
$$\varphi(z_1 \cdots z_n) = \sum_{\pi \in BNC(\chi)}\prod_{B\in\pi}\kappa_{\chi|_B}\paren{(z_1, \ldots, z_n)|_B}.$$
It is possible to explicitly define the cumulants in terms of moments using a M\"obius inversion on the lattice of bi-non-crossing partitions, but these details are not needed now.
For $B = \set{i_1 < \cdots < i_p}$ we denote $z_B := z_{i_1}\cdots z_{i_p}$.
Then for $\pi \in BNC(\chi)$ and for $P \subseteq \pi$ we define
$$\varphi_\pi(z_1, \ldots, z_n) = \prod_{B\in \pi} \varphi(z_B), \qquad \varphi_P(z_1, \ldots, z_n) = \prod_{B \in P}\varphi(z_B).$$
We use a similar notational convention for $\kappa_\pi$, so we may, for example, write $\varphi(z_1\cdots z_n) = \sum_{\pi\in BNC(\chi)}\kappa_\pi(z_1, \ldots, z_n)$.
The main result of \cite{2014arXiv1403.4907C} was the following theorem.

\begin{thm}[\cite{2014arXiv1403.4907C}*{Theorem 4.3.1}]
	Let $\paren{(\A_\ell^{(\iota)}, \A_r^{(\iota)})}_{\iota\in\mathcal{I}}$ be a family of pairs of faces.
	Then the family is bi-free if and only if all mixed cumulants vanish.
	That is, whenever $\chi:\set{1, \ldots, n} \to \set{\ell, r}$, $\e:\set{1, \ldots, n}\to\mathcal{I}$ is non-constant, and $z_1, \ldots, z_n \in \A$ have $z_i \in \A_{\chi(i)}^{(\e(i))}$, one has
	$$\kappa_\chi(z_1, \ldots, z_n) = 0.$$
\end{thm}

\section{Vaccine: a property equivalent to bi-freeness.}
\label{sec:vaccine}
\begin{defn}
	Let $\paren{(\A_\ell^{(\iota)}, \A_r^{(\iota)})}_{\iota \in \mathcal{I}}$ be a family of pairs of faces in a non-commutative probability space $(\A, \varphi)$.
	We say the family has the \emph{vanishing alternating centred $\chi$-interval Eigenschaft\footnote{We use the German word ``Eigenschaft'', meaning ``property'', both to commemorate the fact that this note was partially developed in Germany, and to make the acronym nice.}} (which we will abbreviate as \emph{\blob}) if whenever:
	\begin{itemize}
		\item $n \geq 1$,
		\item $\chi : \set{1, \ldots, n} \to \set{\ell, r}$,
		\item $\epsilon : \set{1, \ldots, n} \to \mathcal{I}$, and
		\item $z_1, \ldots, z_n \in \A$ are such that:
			\begin{itemize}
				\item $z_i \in \A_{\chi(i)}^{(\epsilon(i))}$; and
				\item whenever $\set{i_1 < \cdots < i_k}$ is a maximal $\e$-monochromatic $\chi$-interval, $\varphi(z_{i_1}\cdots z_{i_k}) = 0$,
			\end{itemize}
	\end{itemize}
	it follows that $\varphi(z_1\cdots z_n) = 0$.
\end{defn}

\begin{exam}
	\label{ex:vaccine}
Suppose $\chi$ is such that $\chi^{-1}(\ell) = \set{2,5,6,7,9}$, $\chi^{-1}(r) = \set{1, 3, 4, 8, 10}$, and $\e$ corresponds to the colouring below (i.e.,
$\e^{-1}\paren{\makeaball{0}} = \set{1,2,3,5,8,9,10}$ and
$\e^{-1}\paren{\makeaball{1}} = \set{4,6,7}$).
	
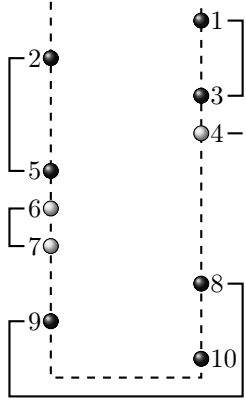
\[\begin{tikzpicture}[baseline]
		\draw[thick, dashed] (-1,0.25) -- (-1, -4.75) -- (1,-4.75) -- (1,0.25);

		\def\colours{{0, 0, 0, 1, 0, 1, 1, 0, 0, 0}}
		\def\sidez{{1,-1,1,1,-1,-1,-1,1,-1,1}}
		\foreach \y in {0,...,9} {
			\pgfmathtruncatemacro{\nodename}{\y+1}
			\pgfmathtruncatemacro{\sd}{\sidez[\y]}
			\pgfmathparse{\palette[\colours[\y]]}
			\def\clr{\pgfmathresult}
			\node (ball\nodename) [draw, shade, circle, ball color=\clr, inner sep=0.07cm] at (\sd, -\y*0.5) {};
			\ifthenelse{\sd=1}{\node[right] at (\sd, -\y*0.5) {\nodename}}
					{\node[left] at (\sd, -\y*0.5) {\nodename}};
		}
		\draw[rline,thick] (ball2)++(-0.35,0) -- ++(-0.2,0) |- ($ (ball5) + (-0.35,0) $);
		\draw[rline,thick] (ball6)++(-0.35,0) -- ++(-0.2,0) |- ($ (ball7)+(-0.35,0) $);
		\draw[rline,thick] (ball9)++(-0.35,0) -- ++(-0.2,0) |- (0,-5) -| ($ (ball8)+(0.55,0) $) -- ++(-0.2,0);
		\draw[rline,thick] (ball4)++(0.35,0) -- ++(0.2,0);
		\draw[rline,thick] (ball3)++(0.35,0) -- ++(0.2,0) |- ($ (ball1)+(0.35,0) $);
	\end{tikzpicture}\]
	The maximal $\e$-monochromatic $\chi$-intervals are $\set{2, 5}, \set{6,7}, \set{8,9,10}, \set{4}$, and $\set{1,3}$.
	Vaccine would imply $\varphi(z_1\cdots z_{10}) = 0$ whenever $z_1, \ldots, z_{10}$ are chosen corresponding to $\chi$ and $\e$ with $$0 = \varphi(z_2z_5) = \varphi(z_6z_7) = \varphi(z_8z_9z_{10}) = \varphi(z_4) = \varphi(z_1z_3).$$
	The reason this condition becomes more complicated than in the free case amounts to the fact that we cannot replace $z_1z_3$ or $z_8z_9z_{10}$ with single elements of either the left or right faces; in the latter case, because neither face may contain an appropriate operator, and in the former because $z_2$ may not commute with $z_3$ or $z_1$ and so the two may not be moved next to each other.

\end{exam}

\subsection{The equivalence.}
\begin{lem}
	\label{lem:bifreeimpliesvaccine}
	Let $\paren{(\A_\ell^{(\iota)}, \A_r^{(\iota)})}_{\e \in \mathcal{I}}$ be a family of pairs of faces in a non-commutative probability space $(\A, \varphi)$.
	Then the family has \blob{} if the pairs of faces are bi-free.
\end{lem}

\begin{proof}
	Let $n \geq 1$, $\chi : \set{1, \ldots, n} \to \set{\ell, r}$, and $\epsilon : \set{1, \ldots, n} \to \mathcal{I}$, and denote by $\mathcal{J}$ the set of maximal $\e$-monochromatic $\chi$-intervals in $\set{1, \ldots, n}$.
	Note that $\mathcal{J} \in BNC(\chi)$ may be thought of as a bi-non-crossing partition in its own right, which will sometimes be of use notationally.
	For $P \subset \set{1, \ldots, n}$, let $m(P)$ denote the minimum element of $BNC(\chi)$ containing $P$ as a block, so all blocks of $m(P)$ except $P$ are singletons.
	Let $b : \set{\pi \in BNC(\chi) : \pi \leq \e} \to \mathcal{J}$ be a function with the following properties:
	\begin{itemize}
		\item if $\pi \in BNC(\chi)$, $j \in b(\pi)$, and $j \sim_\pi k$, then $k \in b(\pi)$ (i.e., the interval $b(\pi)$ is isolated in $\pi$: $\pi \leq \set{b(\pi), b(\pi)^c}$); and
		\item if $\pi, \sigma \in BNC(\chi)$ satisfy $\pi\vee m(b(\pi)) = \sigma \vee m(b(\pi))$ then $b(\pi) = b(\sigma)$ (i.e., any partition obtained from $\pi$ by only modifying the part of $\pi$ in $b(\pi)$ is mapped to the same $\chi$-interval by $b$).
	\end{itemize}
	For example, one could take $b(\pi)$ to be the $\chi$-minimal element of $\mathcal{J}$ which is isolated in $\pi$.
	Any partition $\pi \in BNC(\chi)$ with $\pi \leq \e$ must leave one element of $\mathcal{J}$ isolated; indeed, if one takes $\pi \vee \mathcal{J} \leq \e$ (the element of $BNC(\chi)$ obtained from $\pi$ by joining all points lying in the same $\e$-monochromatic $\chi$-intervals), it must contain a $\chi$-interval (as any non-crossing partition, in particular $s_\chi^{-1}\cdot \pi\vee\mathcal{J}$, must contain an interval) and since $\mathcal{J} \leq \pi \vee \mathcal{J} \leq \e$, any interval it contains must be isolated and maximal $\e$-monochromatic.
	Then this same interval is isolated in $\pi$.

	Denote $S(B) = \set{\pi \in BNC(\chi) : B \in \pi}$ the set of bi-non-crossing partitions in which $B$ is a block.
	Note that if $\sigma \in S(B)$ and $\rho \in S(B^c)$, then $\sigma \wedge \rho$ is a partition with blocks under $B$ corresponding to $\rho$ and blocks outside of $B$ corresponding to $\sigma$.
	Further, any partition $\pi \in BNC(\chi)$ with $\pi \leq \set{B, B^c}$ may be expressed in this form: $\pi \vee m(B) \in S(B)$, $\pi \vee m(B^c) \in S(B^c)$, and $(\pi \vee m(B))\wedge(\pi\vee m(B^c)) = \pi$.

	Now, let $z_1, \ldots, z_n$ be as in the definition of vaccine.
	Using the moment-cumulant formula and the vanishing of mixed cumulants from bi-freeness, we have
	\begin{align*}
		\varphi(z_1\cdots z_n)
		&= \sum_{\substack{\pi \in BNC(\chi)\\ \pi \leq \e}} \kappa_\pi(z_1, \ldots, z_n)\\
		&= \sum_{B \in \mathcal{J}} \sum_{\pi \in b^{-1}(B)} \kappa_\pi(z_1, \ldots, z_n)\\
		&= \sum_{B \in \mathcal{J}} \sum_{\substack{\sigma \in S(B)\\ b(\sigma) = B}}\sum_{\rho \in S(B^c)} \kappa_{\sigma\wedge\rho}(z_1, \ldots, z_n)\\
		&= \sum_{B \in \mathcal{J}} \sum_{\substack{\sigma \in S(B)\\ b(\sigma) = B}}\paren{\sum_{\rho \in BNC(\chi|_B)}\kappa_\rho(z_1, \ldots, z_n)} \kappa_{\sigma\setminus \set{B}}(z_1, \ldots, z_n)\\
		&= \sum_{B \in \mathcal{J}} \sum_{\substack{\sigma \in S(B)\\ b(\sigma) = B}}\varphi(z_B) \kappa_{\sigma\setminus \set{B}}(z_1, \ldots, z_n)\\
		&= 0.
	\end{align*}
	Essentially, in the sum of cumulants representing $\varphi(z_1\cdots z_n)$, we have grouped terms together by isolated intervals, and used the fact that when we sum over the entire lattice of bi-non-crossing partitions over one of these intervals, we recover the moment corresponding to that interval, which is zero by assumption.
\end{proof}

\begin{lem}
	\label{lem:vaccineimpliesbifree}
	Let $\paren{(\A_\ell^{(\iota)}, \A_r^{(\iota)})}_{\iota \in \mathcal{I}}$ be a family of pairs of faces in a non-commutative probability space $(\A, \varphi)$.
	Then the pairs of faces are bi-free if the family has \blob{}.
\end{lem}

\begin{proof}
	We will show that vaccine uniquely specifies mixed moments in terms of pure ones.
	Let $n \geq 1$, $\chi : \set{1, \ldots, n} \to \set{\ell, r}$, $\e : \set{1, \ldots, n} \to \mathcal{I}$, and suppose $z_1, \ldots, z_n \in \A$ with $z_i \in \A_{\chi(i)}^{(\epsilon(i))}$.
	Denote by $\mathcal{J}$ the set of maximal $\e$-monochromatic $\chi$-intervals in $\set{1, \ldots, n}$.

	For each $I = \set{a_1 < \cdots < a_j} \in \mathcal{J}$, let $\lambda_I$ be a (complex) root of the polynomial $\varphi\paren{(z_{a_1} - w)\cdots (z_{a_j} - w)}$.
	Then if $f : \set{1, \ldots, n} \to \mathcal{J}$ is the unique map so that $i \in f(i)$ for every $i$, we have
	$$\varphi\paren{(z_1 - \lambda_{f(1)})\cdots(z_n-\lambda_{f(n)})} = 0,$$
	as the $\lambda$'s were chosen precisely to make the vaccine property apply.
	Expanding this equation gives us an expression for $\varphi(z_1\cdots z_n)$ in terms of mixed moments with at most $n-1$ terms; by recursively applying the same procedure we find an expression for $\varphi(z_1\cdots z_n)$ in terms of pure moments.

	Now, for $\iota \in \mathcal{I}$ let $\varphi^{(\iota)}$ be the restriction of $\varphi$ to $\ang{\A_\ell^{(\iota)}, \A_r^{(\iota)}}$, and $\mu = \substack{**\\ \iota \in \mathcal{I}}\varphi^{(\iota)}$ the bi-free product distribution, which by Lemma~\ref{lem:bifreeimpliesvaccine} also has vaccine.
	We then find that the same expressions for joint moments in terms of pure ones hold under $\mu$ as under $\varphi$, which is to say that $\paren{(\A_\ell^{(\iota)}, \A_r^{(\iota)})}_{\iota\in\mathcal{I}}$ are bi-free.
\end{proof}

\begin{thm}
	Let $\paren{(\A_\ell^{(\iota)}, \A_r^{(\iota)})}_{\iota \in \mathcal{I}}$ be a family of pairs of faces in a non-commutative probability space $(\A, \varphi)$.
	Then the family has \blob{} if and only if the pairs of faces are bi-free.
\end{thm}

	We pause here to comment that although our proof of Lemma~\ref{lem:bifreeimpliesvaccine} relied on the bi-free cumulants and the behaviour of bi-non-crossing partitions, it is possible to establish the result without them; our motivation in using them was largely to simplify notation and avoid introducing the construction of a bi-free representation on a free product of Hilbert spaces from \cite{voiculescu2014free} or the structure of $LR$-diagrams from \cite{2014arXiv1403.4907C}.
	However, the same result may be obtained from the following proposition.

	\begin{prop}
		Suppose that $(\X_{\iota}, \xi_{\iota}, \mathring\X_{\iota})_{\iota\in\mathcal{I}}$ is a collection of vector spaces with specified state vectors (i.e., $\X_{\iota}$ is a vector space, $\xi_{\iota} \in \X_{\iota}$ a non-zero vector and $\mathring\X_{\iota}\subset\X_{\iota}$ a subspace so that $\X_{\iota} = \C\xi_{\iota} \oplus \mathring\X_{\iota}$) and let $(\X, \xi, \mathring\X) = \substack{*\\\iota \in \mathcal{I}}(\X_{\iota}, \xi_\iota, \mathring\X_\iota)$ be their free product.
		Let $\lambda_\iota$ and $\rho_\iota$ be the left and right representations of $\mathcal{L}(\X_\iota)$ in $\mathcal{L}(\X)$.

		Fix $n \in \N$ and take $\chi : \set{1, \ldots, n} \to \set{\ell, r}$ and $\epsilon : \set{1, \ldots, n} \to \mathcal{I}$.
		Choose $z_i \in \mathcal{L}(\X_{\epsilon(i)})$ so that whenever $\set{i_1 < \cdots < i_k}$ is a maximal $\epsilon$-monochromatic $\chi$-interval, $z_{i_1}\cdots z_{i_k}\xi_{\epsilon(i)} \in \mathring\X_{\epsilon(i)}$.
		Set $\mu_i = \lambda_{\epsilon(i)}$ if $\chi(i) = \ell$ and $\mu_i = \rho_{\epsilon(i)}$ if $\chi(i) = r$.
		Then if the maximal $\epsilon$-monochromatic $\chi$-intervals are denoted $J_1, \ldots, J_m$ in $\prec_\chi$-increasing order, and $\iota_i$ is taken to be the colour under $\epsilon$ of the $i$-th interval,
		$$
		\mu_1(z_1)\cdots\mu_n(z_n)\xi = z_{J_1}\xi_{\iota_1}\otimes z_{J_2}\xi_{\iota_2} \otimes \cdots \otimes z_{J_m}\xi_{\iota_m} \in \mathring\X.
		$$
	\end{prop}

	This proposition may be established by using a combinatorial argument similar to the one in the original proof of Lemma~\ref{lem:bifreeimpliesvaccine} and the calculus with $LR$-diagrams introduced in \cite{2014arXiv1403.4907C}: one groups diagrams in which not all spines reach the top of the diagram by a selected isolated interval, and then argues that summing over all diagrams with the same interval isolated yields a contribution of the moment of that interval, which is zero, meaning the only surviving diagrams are ones where every interval has a spine extending to the top of the diagram; a bit of algebra from there shows that the contribution of such diagrams together is the claimed vector.
	We provide a more direct argument here.

	\begin{proof}
		We will prove a related claim.
		Let us take $V_\iota$ and $W_\iota$ to be the left and right factoring isomorphisms used in defining the free product representations in \cite{voiculescu2014free}, so
		$$V_\iota : \X_\iota \otimes\paren{\C\xi \oplus\bigoplus_{a\geq1}\bigoplus_{\iota\neq i_1\neq\cdots\neq i_a}\mathring\X_{i_1}\otimes\cdots\mathring\X_{i_a} } \to \X \qquad\text{and}\qquad W_\iota : \paren{\C\xi \oplus\bigoplus_{a\geq1}\bigoplus_{i_1\neq\cdots\neq i_a \neq \iota}\mathring\X_{i_1}\otimes\cdots\mathring\X_{i_a} } \otimes \X_\iota \to \X.$$
		We will show that if $z_1, \ldots, z_n$ are as above but we only apply the assumption that $z_J\xi_{\iota} \in \mathring\X_\iota$ for $\chi$-intervals which are neither the first nor the last, then
		$$V_{\iota_1}^{-1}(\mu_1(z_1)\cdots\mu_n(z_n)\xi)
		= z_{J_1}\xi_{\iota_1}\otimes \paren{z_{J_2}\xi_{\iota_2}\otimes\cdots\otimes z_{J_{m-1}}\xi_{\iota_{m-1}}\otimes \mathring z_{J_m}\xi_{\iota_m} + z_{J_2}\xi_{\iota_2}\otimes\cdots\otimes z_{J_{m-1}}\xi_{\iota_{m-1}}\varphi_{\iota_m}(z_{J_m})},$$
		where an empty tensor product is interpreted as $\xi$ (in the case that $m = 2$), $\varphi_{\iota}$ is the state on $\mathcal{L}(\X_\iota)$ such that $T\xi_\iota - \varphi(T)\xi_\iota \in \mathring\X_\iota$, $\mathring z_{J} = z_J - \varphi_\iota(z_J)$, and the entire right tensor factor is replaced by $\xi$ if $\epsilon$ is constant.
		A similar result holds for $W_{\iota_m}^{-1}$.


		Notice that if there is only one operator, the claim holds immediately.
		Let us therefore proceed by induction on $n$, the number of operators.
		We will assume that $\chi(1) = r$ as the proof when $\chi(1) = \ell$ is similar; notice that this means $\iota_m = \epsilon(1)$.

		Suppose that $J$ is the $\chi$-maximal $\epsilon$-monochromatic $\chi$-interval of $\set{2, \ldots, n}$.
		There are two cases: first, suppose that $J$ has the colour $\epsilon(1)$.
		In that case, using our inductive hypothesis,
		\begin{align*}\rho_{\iota_m}(z_1)\mu_2(z_2)\cdots\mu_n(z_n)\xi
			&= W_{\iota_m}(1\otimes z_1)W_{\iota_m}^{-1}(\mu_2(z_2)\cdots\mu_n(z_n)\xi) \\
			&= W_{\iota_m}(1\otimes z_1)(\eta \otimes z_J) \\
			&= W_{\iota_m}(\eta \otimes z_1z_J),
		\end{align*}
		for the appropriate $\eta$; this is what we have aimed to show.
		On the other hand, suppose that $J$ has a colour other than $\epsilon(1)$.
		Then either $z_J\xi_{m-1} \in \mathring\X_{\iota_{m-1}}$ or $m = 2$ and $J$ is the $\chi$-first $\epsilon$-monochromatic $\chi$-interval.
		In either case, we find
		$$W_{\iota_m}^{-1}(\mu_2(z_2)\cdots\mu_n(z_n)\xi)
		= \paren{\mathring z_{J_1}\xi_{\iota_1}\otimes z_{J_2}\xi_{\iota_2}\otimes\cdots\otimes z_{J_{m-1}}\xi_{\iota_{m-1}} + \varphi_{\iota_1}(z_{J_1})z_{J_2}\xi_{\iota_2}\otimes\cdots\otimes z_{J_{m-1}}\xi_{\iota_{m-1}}} \otimes \xi_{\iota_m}.$$
		Hence acting by $(id\otimes z_1)$ produces the result we claimed.

		The proposition now follows by looking at the particular case when the first and last $\chi$-intervals also meet the assumptions.
	\end{proof}

	Yet another approach to the proof of Lemma~\ref{lem:bifreeimpliesvaccine} was pointed out to us by Ping Zhong of the University of Waterloo not long after this note appeared on the ar$\chi$iv.
	One can show that the left and right representations of operators on the free product vector space commute provided either that they come from different coloured algebras (this was remarked by Voiculescu in \cite{voiculescu2014free}) or that they are not in the same maximal $\epsilon$-monochromatic $\chi$-interval, there is at least one such interval $\chi$-between them, and all such intervals are centred.
	Using this one can assume that all maximal $\epsilon$-monochromatic $\chi$-intervals except for the one containing $n$ are singletons, which makes a direct computation tractable.

We will now make a digression to examine how this argument extends to other settings in bi-free probability.
Each of the following subsections is self-contained and not used in the rest of the paper, and the reader may safely skip to the beginning of Section~\ref{sec:liberation}.

\subsection{Conditional bi-freeness.}
Conditional bi-freeness was studied by Gu and Skoufranis in \cite{gu2016conditionally}, building off of conditional free independence which was introduced by Bo\.zejko, Leinert, and Speicher \cite{bozejko1996convolution}.
We will show that conditional bi-free independence also admits a characterization in terms of $\chi$-intervals.
First, though, we take the time to introduce some notation.

Suppose that $\pi \in BNC(\chi)$.
A block $B \in \pi$ is said to be \emph{inner} if there is another block $C \in \pi$ and $j, k \in C$ so that for every $i \in B$, $j \prec_\chi i \prec_\chi k$; a block which is not inner is said to be \emph{outer}.
With $\pi$ as in Example~\ref{ex:partition}, $\set{2,5,7}$ and $\set{1}$ are outer while $\set{3,4}$ and $\set{6,8}$ are inner.

Let $(\A, \varphi)$ be a non-commutative probability space, and $\theta$ a state on $\A$.
The conditional cumulants with respect to $(\theta, \varphi)$ are multilinear functionals $\mc{K}_\chi : \A^n \to \C$ defined by the requirement that for any $z_1, \ldots, z_n \in \A$,
$$\theta(z_1\cdots z_n) = \sum_{\pi\in BNC(\chi)}\paren{\prod_{\substack{V\in\pi\\V \text{ inner}}} \kappa_{\chi|_V}\paren{(z_1, \ldots, z_n)|_V}}\paren{\prod_{\substack{V\in\pi\\V\text{ outer}}}\mc{K}_{\chi|_V}\paren{(z_1, \ldots, z_n)|_V}}.$$
Here $\kappa$ represents the usual bi-free cumulants taken with respect to $\varphi$.
For $\pi \in BNC(\chi)$, we will denote by $\mc{K}_\pi(z_1, \ldots, z_n)$ the term in the above sum corresponding to $\pi$, a product of $\kappa_{\chi|_V}$ terms and $\mc{K}_{\chi|_V}$ terms.
We will say a family $\paren{\A_\ell^{(\iota)}, \A_r^{(\iota)}}_{\iota\in\I}$ is \emph{conditionally bi-free} in $(\A, \theta, \varphi)$ if it is bi-free with respect to $\varphi$ and all mixed conditional cumulants vanish; it was shown in \cite{gu2016conditionally} that this is equivalent to their definition in terms of free product representations, and moreover, that being conditionally bi-free uniquely specifies the mixed $\theta$-moments in terms of the pure $\theta$-moments and $\varphi$-moments.

\begin{thm}
	Let $(\A, \varphi)$ be a non-commutative probability space and $\theta : \A \to \C$ a state on $\A$.
	Suppose $\paren{\A_\ell^{(\iota)}, \A_r^{(\iota)}}_{\iota\in\I}$ is a family of pairs of faces in $\A$.
	Then the family is conditionally bi-free if and only if whenever:
	\begin{itemize}
		\item $n \geq 1$,
		\item $\chi : \set{1, \ldots, n} \to \set{\ell, r}$,
		\item $\epsilon : \set{1, \ldots, n} \to \mathcal{I}$,
		\item $\mc{J}$ is the set of maximal $\e$-monochromatic $\chi$-intervals, and
		\item $z_1, \ldots, z_n \in \A$ are such that:
			\begin{itemize}
				\item $z_i \in \A_{\chi(i)}^{(\epsilon(i))}$; and
				\item $\varphi(z_J) = 0$ for each $J \in \mc{J}$
			\end{itemize}
	\end{itemize}
	it follows that
	$$\varphi(z_1\cdots z_n) = 0 \qquad\text{and}\qquad \theta(a_1\cdots a_n) = \prod_{J \in \mc{J}} \theta(z_J).$$
\end{thm}

\begin{proof}
	By the same argument as in the proof of Lemma~\ref{lem:vaccineimpliesbifree} it follows that the conditions assumed above suffice to uniquely specify all mixed $\varphi$- and $\theta$-moments in terms of pure $\varphi$- and $\theta$-moments; hence if we can show that conditionally bi-free families satisfy this condition the proof will be complete.
	The condition on $\varphi$ is precisely vaccine, so we need only show that our expression for mixed $\theta$-moments is correct.
	
	We take an approach similar to that of Lemma~\ref{lem:bifreeimpliesvaccine} for deducing the value of $\theta$.
	We claim that the only terms which contribute to the value of $\theta$ in the cumulant expansion are those corresponding to partitions $\pi < \mathcal{J}$, where once again $\mathcal{J}$ is the set of maximal $\e$-monochromatic $\chi$-intervals.
	Towards this end, let $b : \set{\pi \in BNC(\chi) : \pi \leq \e} \to \mc{J}$ be as in Lemma~\ref{lem:bifreeimpliesvaccine}, with the additional constraint that $b$ picks interior intervals whenever $\pi \nleq \mc{J}$.
	That is, $b$ should have the following properties:
	\begin{itemize}
		\item if $\pi \in BNC(\chi)$, $j \in b(\pi)$, and $j \sim_\pi k$, then $k \in b(\pi)$ (i.e., the interval $b(\pi)$ is isolated in $\pi$: $\pi \leq \set{b(\pi), b(\pi)^c}$);
		\item if $\pi, \sigma \in BNC(\chi)$ satisfy $\pi\vee m(b(\pi)) = \sigma \vee m(b(\pi))$ then $b(\pi) = b(\sigma)$ (i.e., any partition obtained from $\pi$ by only modifying the part of $\pi$ in $b(\pi)$ is mapped to the same $\chi$-interval by $b$); and
		\item if $\pi \in BNC(\chi)$ and $\pi \nleq \mc{J}$, then $b(\pi)$ is an inner block in $\pi \vee \mc{J}$.
	\end{itemize}
	Such functions exist: for example, one could take $b(\pi)$ to be the $\chi$-minimal element of $\mc{J}$ which is inner and isolated in $\pi$, if such exists, and the $\chi$-minimal element of $\mc{J}$ otherwise.
	Note that if $\pi \nleq \mc{J}$, $\pi$ must connect two intervals in $\mc{J}$ and so there must be an inner block in $\mc{J}\vee\pi$ between these two intervals.
	As before, let $S(B) = \set{\pi\in BNC(\chi) : B \in \pi}$, and set $S_i(B) = \set{\pi\in S(B) : B \text{ inner in } \pi\vee\mc{J}}$.
	We now compute much as in the proof of Lemma~\ref{lem:bifreeimpliesvaccine}.
	Supposing $z_1, \ldots, z_n$ meet the hypotheses of the lemma:
	\begin{align*}
		\theta(z_1\cdots z_n)
		&= \sum_{\substack{\pi\in BNC(\chi)\\ \pi\leq \e}} \mc{K}_\pi(z_1, \ldots, z_n)\\
		&= \sum_{B \in \mc{J}}\paren{\sum_{\substack{\pi\in b^{-1}(B)\\ B \text{ inner in } \pi\vee\mc{J}}}\mc{K}_\pi(z_1, \ldots, z_n) + \sum_{\substack{\pi\in b^{-1}(B)\\ B \text{ outer in } \pi\vee\mc{J}}}\mc{K}_\pi(z_1, \ldots, z_n)}\\
		&= \sum_{B\in\mc{J}} \paren{\sum_{\pi \in S_i(B) \cap b^{-1}(B)} \varphi(z_B)\mc{K}_{\pi\setminus\set{B}}(z_1, \ldots, z_n) + \sum_{\substack{\pi\in b^{-1}(B)\\ B \text{ outer in } \pi\vee\mc{J}}}\mc{K}_\pi(z_1, \ldots, z_n)}\\
		&= \sum_{\substack{\pi\in b^{-1}(B)\\ B \text{ outer in } \pi\vee\mc{J}}}\mc{K}_\pi(z_1, \ldots, z_n)\\
		&= \sum_{\substack{\pi \in BNC(\chi)\\\pi\leq\mc{J}}} \mc{K}_\pi(z_1, \ldots, z_n)\\
		&= \prod_{J\in\mc{J}}\sum_{\pi_J \in BNC(\chi|_J)} \mc{K}_{\pi_J}\paren{(z_1, \ldots, z_n)|_J}\\
		&= \prod_{J\in\mc{J}}\theta(z_J).
	\end{align*}
	Here in the last few lines we have noted that summing over all partitions sitting under $\mc{J}$ is the same as summing over partitions sitting under each interval individually, and then taking the product; this is valid since every term in $\mc{K}_\pi$ is a product of terms corresponding to blocks, and each block must be contained in a single interval in $\mc{J}$.
\end{proof}

\subsection{Bi-freeness with amalgamation.}
	Bi-free independence with amalgamation was introduced by Voiculescu in \cite{voiculescu2014free}, and further studied by Nelson, Skoufranis, and the author in \cite{Charlesworth:2015aa}.
	The setting is that of a $\B$-$\B$-non-commutative probability space, which is a triple $(\A, E, \varepsilon)$ where $\A$ is a unital $*$-algebra, $\varepsilon : \B\otimes \B^{op} \to \A$ is a unital homomorphism which is injective on $\B\otimes 1$ and $1\otimes \B^{op}$, and $E : \A\to \B$ is a linear map so that for $b_1, b_2 \in \B$ and $T \in \A$, we have
	$$E(\varepsilon(b_1\otimes b_2)T) = b_1 E(T) b_2 \qquad\text{and}\qquad E(T\varepsilon(b_1\otimes 1)) = E(T\varepsilon(1\otimes b_1)).$$
	Then bi-freeness with amalgamation can be defined either in terms of moments under $E$ matching those coming from an abstract free product with amalgamation over $\B$ of $\B$-$\B$-bimodules with specified $\B$-vector states, or by the vanishing of mixed bi-multiplicative $\B$-valued cumulants.
	
	In this setting the proof of Lemma~\ref{lem:bifreeimpliesvaccine} goes through with additional bookkeeping required. The combinatorial idea still holds, that one can discover the moment corresponding to an isolated interval in each term in the expression of $E(z_1\cdots z_n)$ in terms of cumulants, although one must take care to account for the nature of bi-multiplicative functions since the range of $E$ is now $\B$ rather than the scalars.

	The proof of Lemma~\ref{lem:vaccineimpliesbifree} requires a bit more care, however, essentially due to the fact that $\B$ is not necessarily algebraically closed and due to its non-commutativity $E$ does not necessarily map polynomials with coefficients in $\A$ to polynomials with coefficients in $\B$; essentially, the variables may become trapped between coefficients in such a way that they cannot be pulled out of the $E$.
	We do have the following Lemma, however, which will allow us to prove an analogue of Lemma~\ref{lem:vaccineimpliesbifree} in many settings.

	\begin{lem}
		Suppose $(\A, E, \varepsilon)$ is a $\B$-$\B$-non-commutative probability space, with $\B$ a Banach algebra.
		Then for every $\chi : \set{1, \ldots, n} \to \set{\ell, r}$ and $z_i \in \A_{\chi(i)}$ there exist $\hat b_i\in \B$ so that, with $b_i = \varepsilon(\hat b_i\otimes 1)$ if $\chi(i) = \ell$ and $b_i = \varepsilon(1\otimes \hat b_i)$ if $\chi(i) = r$, we have
		$$E((z_1-b_1)\cdots(z_n-b_n)) = 0.$$
	\end{lem}
	
	\begin{proof}
		Let $j = \min_{\prec_\chi}\paren{\set{1,\ldots, n}}$.
		Notice that we can write
		\begin{align*}
			E((z_1-b_1)\cdots(z_n-b_n)) =
			&E( (z_1-b_1)\cdots (z_{j-1}-b_{j-1})z_j(z_{j+1}-b_{j+1})\cdots(z_n-b_n)) \\
			&\qquad - \hat b_jE( (z_1-b_1)\cdots(z_{j-1}-b_{j-1})(z_{j+1}-b_{j+1})\cdots(z_n-b_n)).
		\end{align*}
		Indeed, this is immediate if $\chi(j) = \ell$ (since $\chi(k) = r$ for $k < j$ in this case), while if $\chi(j) = r$ it must be that $j = n$, so replacing $b_n$ by $\varepsilon(\hat b_n\otimes 1)$ in the product does not affect the value of the expectation, and $\varepsilon(\hat b_n\otimes 1)$ can then be pulled past all of the right operators and out the left.
		Now, if we take $b_i = \lambda \in \C$ for all $i \neq j$, we find that $b_i$ commutes with every $z_k$, and
		$$E( (z_1-\lambda)\cdots(z_{j-1}-\lambda)(z_{j+1}-\lambda)\cdots(z_n-\lambda)) = (-\lambda)^n + \mathcal O\paren{\lambda^{n-1}},$$
		which is a polynomial in $\lambda$ with coefficients in $\B$ and leading term $(-\lambda)^n$.
		In particular, for $\lambda$ sufficiently large it is invertible in $\B$.
		Then once $\lambda$ is large enough, we may take
		\begin{align*}
			  \hat b_j
			  &= {E\paren{(z_1-\lambda)\cdots(z_{j-1}-\lambda)z_j(z_{j+1}-\lambda)\cdots(z_n-\lambda)}} \\
			  &\qquad \cdot{E\paren{(z_1-\lambda)\cdots(z_{j-1}-\lambda)(z_{j+1}-\lambda)\cdots(z_n-\lambda)}}^{-1},
		  \end{align*}
		  producing a solution to our equation.
	\end{proof}

	For the above lemma, we needed something weaker than $\B$ being a Banach algebra: we only require that monic polynomials with coefficients in $\B$ are invertible when evaluated at least one complex number.
	With this lemma in hand, we can reprove Lemma~\ref{lem:vaccineimpliesbifree} in the amalgamated setting; the only difference is that for each maximal $\chi$-interval $I$ we must choose a solution to an equation with $\abs{I}$ variables rather than only one.
	We therefore have the following theorem:

	\begin{thm}
		Suppose that $\B$ is a Banach algebra, and let $\fpf$ be a family of pairs of $\B$-faces in a $\B$-$\B$-non-commutative probability space $(\A, E, \varepsilon)$.
		Then the family has vaccine if and only if the pairs of $\B$-faces are bi-free with amalgamation over $\B$.
	\end{thm}

\section{A liberation condition for bi-freeness.}
\label{sec:liberation}
Our aim in this section is to define a multiplicative bi-free Brownian motion, as an analogue to the free unitary Brownian motion introduced by Biane \cite{biane1997free}.
Many related results in the free case were obtained in the context of a tracial von Neumann algebra, allowing the arguments to be simplified; unfortunately that luxury is not available to us in the context of bi-free probability as we are not aware of an appropriate analogue of traciality.
As the following example demonstrates, simply asking that the state on the non-commutative probability space be tracial is too restrictive.
\begin{exam}
	Suppose $\paren{\A_\ell^{(\iota)}, \A_r^{(\iota)}}_{\iota\in\set{\makeaball{0},\makeaball{1}}}$ are bi-free pairs of faces in a non-commutative probability space $(\A, \varphi)$.
	Then we have for $x \in \A_\ell^{(\makeaball{0})}$, $w \in \A_r^{(\makeaball{0})}$, $y \in \A_\ell^{(\makeaball{1})}$, and $z \in \A_r^{(\makeaball{1})}$ that
	\[
	\begin{tikzpicture}[baseline]
		\def\colours{{0, 0, 1, 1, 0}}
		\def\sidez{{1, -1,-1,1,1}}
		\def\opnames{{"$w$", "$x$", "$y$", "$z$", "$w$"}}

		\begin{scope}[shift={(-\textwidth*0.25,0)}]
			\path (0,0)--(-\textwidth*0.25,0);
		\draw[thick, dashed] (-1,0.25) -- (-1, -1.75)
		-- node [below, scale=0.7]{$\varphi(xyzw) = \varphi(xw)\varphi(y)\varphi(z) + \varphi(x)\varphi(w)\varphi(yz) - \varphi(x)\varphi(w)\varphi(y)\varphi(z)$}
		(1,-1.75) -- (1,0.25);

		\foreach \y in {0,...,3} {
			\pgfmathparse{\opnames[1+\y]}
			\edef\nodename{\pgfmathresult}
			\pgfmathtruncatemacro{\sd}{\sidez[1+\y]}
			\pgfmathparse{\palette[\colours[1+\y]]}
			\def\clr{\pgfmathresult}
			\node (ball\nodename) [draw, shade, circle, ball color=\clr, inner sep=0.07cm] at (\sd, -\y*0.5) {};
			\ifthenelse{\sd=1}{\node[right] at (\sd, -\y*0.5) {\nodename}}
					{\node[left] at (\sd, -\y*0.5) {\nodename}};
		}
		\end{scope}
		\node [below, scale=0.7] at (1.25,-1.75) {while};
		\begin{scope}[shift={(\textwidth*0.25,0)}]
			\path (0,0)--(\textwidth*0.25,0);
		\draw[thick, dashed] (-1,0.25) -- (-1, -1.75)
		-- node [below, scale=0.7]{$\varphi(wxyz) = \varphi(wx)\varphi(yz)$.}
		(1,-1.75) -- (1,0.25);

		\foreach \y in {0,...,3} {
			\pgfmathparse{\opnames[\y]}
			\edef\nodename{\pgfmathresult}
			\pgfmathtruncatemacro{\sd}{\sidez[\y]}
			\pgfmathparse{\palette[\colours[\y]]}
			\def\clr{\pgfmathresult}
			\node (ball\nodename) [draw, shade, circle, ball color=\clr, inner sep=0.07cm] at (\sd, -\y*0.5) {};
			\ifthenelse{\sd=1}{\node[right] at (\sd, -\y*0.5) {\nodename}}
					{\node[left] at (\sd, -\y*0.5) {\nodename}};
		}
		\end{scope}
	\end{tikzpicture}
	\]
	Note that these two terms fail to be equal even when $(x, w)$, $(y, z)$ are a bi-free standard semicircular system with $\varphi(wx) = \varphi(yz) = 1$, as the left expression vanishes while the right equals $1$.
\end{exam}

\subsection{Free Brownian motion.}
We take some time to review the concept of free Brownian motion, which is the free analogue of the Gaussian process acting on a Hilbert space.
This is intended to be a brief treatment; a more complete description of free Brownian motion and free stochastic calculus may be found in \cite{voiculescu1992free}.

\begin{defn}
	A \emph{free Brownian motion} in a non-commutative probability space $(\A, \varphi)$ is a non-commutative stochastic process $(S(t))_{t\geq0}$ such that:
	\begin{itemize}
		\item the increments of $S(t)$ are free: for $0\leq t_1 < \cdots < t_k$, the collection $S(t_2)-S(t_1), \ldots, S(t_k)-S(t_{k-1})$ are freely independent; and
		\item the process is stationary, with semicircular increments: for $0 \leq s < t$, $S(t)-S(s)$ is semicircular with variance $t-s$.
	\end{itemize}
\end{defn}

Free Brownian motion can be modelled on a Fock space \cite{voiculescu1992free}. Indeed, suppose
$$\mathcal{F}\paren{L^2(\R_{\geq0})} := \C\Omega \oplus \bigoplus_{n\geq1} L^2(\R_{\geq0})^{\otimes n}.$$
Let $\xi_t = 1_{[0, t]}$, and define $S(t) = l(\xi_t) + l^*(\xi_t)$.
Then $(S(t))_t$ is a free Brownian motion.

Free unitary Brownian motion was initially introduced by Biane in \cite{biane1997free} as a multiplicative analogue of the (additive) free Brownian motion above.
Its definition makes reference to a certain family of measures $(\nu_t)_{t\geq0}$ supported on $\mathbb{T}$, introduced by Bercovici and Voiculescu in \cite{bercovici1992levy}.
In particular, $\nu_t$ has the property that for $t, s \geq 0$, $\nu_t\boxtimes\nu_s = \nu_{t+s}$.
We do not require the particular details of its introduction and so will eschew them.

\begin{defn}
	A \emph{free unitary Brownian motion} in a non-commutative probability space $(\A, \varphi)$ is a non-commutative stochastic process $(U(t))_{t\geq0}$ such that:
	\begin{itemize}
		\item the (left) multiplicative increments of $U(t)$ are free: for $0 \leq t_1 < \cdots < t_k$, the increments given by $U^*(t_1)U(t_2), U^*(t_2)U(t_3), \ldots, U^*(t_{k-1})U(t_k)$ are freely independent; and
		\item the process is stationary with increments prescribed by $\nu_\cdot$: the distribution of $U^*(t)U(s)$ depends only on $s-t$, and is in fact $\nu_{s-t}$.
	\end{itemize}
\end{defn}

It was shown in \cite{biane1997free} that if $S(t)$ is a Fock space realization of a free additive Brownian motion and $U(t)$ the solution to the free stochastic differential equation
$$dU(t) = iU(t)\,dS(t) - \frac12 U(t)\,dt$$
with $U(0) = 1$, then $U(t)$ is a free unitary Brownian motion.
Moreover, the moments of a free unitary Brownian motion were computed: for $n > 0$,
$$\varphi(U(t)^n) = \sum_{k=0}^{n-1}(-1)^k\frac{t^k}{k!}n^{k-1}{n \choose k+1}e^{-nt/2}.$$
A consequence is that free unitary Brownian motion converges in distribution to a Haar unitary, i.e., a unitary $u_\infty$ with $\varphi(u_\infty^{k}) = \delta_{k=0}$ for $k \in \Z$.
Another important results from \cite{biane1997free} is the following bound: for some $K > 0$ and any $t > 0$,
$$\norm{U(t)-e^{-t/2}} \leq K\sqrt{t}.$$

Haar unitaries are important within free probability because conjugating by them can create freeness: if $\A_1, \ldots, \A_n \subset \A$ are free from the Haar unitary $u \in \A$, then $\A_1, u\A_2u, \ldots, u^{n-1}\A_nu^{n-1}$ are free.
In \cite{Charlesworth:2015aa}, the author together with Nelson and Skoufranis identified a bi-free analogue: if $u_\ell, u_r \in \A$ are such that the $*$-distribution of the pair $(u_\ell, u_r)$ is the same as that of $(u, u^*)$ with $u$ a Haar unitary (so, for example, $\varphi(u_\ell^j u_r^k) = \delta_{j=k}$) and $\paren{(\A^{(\iota)}_\ell, \A^{(\iota)}_r)}_{\iota=1}^n$ are pairs of faces in $\A$ bi-free from $(u_\ell, u_r)$, then the faces $\paren{(u_\ell^{\iota-1}\A^{(\iota)}_\ell (u_\ell^*)^{\iota-1}, u_r^{\iota-1}\A^{(\iota)}_r(u_r^*)^{\iota-1})}_{\iota=1}^n$ are bi-free.
We take motivation from this fact to define a bi-free unitary Brownian motion: we want the process to tend to the distribution of a Haar pair of unitaries, so that conjugating by the process asymptotically creates bi-freeness and can therefore be seen as a sort of liberation.

\subsection{The free liberation derivation.}
Suppose that $A, B$ are algebraically free unital sub-algebras generating a tracial non-commutative probability space $(\A, \tau)$.
In \cite{Voiculescu1999101}, Voiculescu defined the derivation $\delta_{A:B} : \A \to \A\otimes\A$ to be a linear map satisfying the Leibniz rule such that $\delta_{A:B}(a) = a\otimes1-1\otimes a$ for $a \in A$ and $\delta_{A:B}(b) = 0$ for $b \in B$.
It was shown that $A$ and $B$ are freely independent if and only if $(\tau\otimes\tau)\circ\delta_{A:B} \equiv 0$.
Moreover, the derivation $\delta_{A:B}$ relates to how the joint distribution of $A$ and $B$ changes as $A$ is perturbed by unitary free Brownian motion.
\begin{prop}[\cite{Voiculescu1999101}*{Proposition 5.6}]
	Let $A, B$ be two unital $*$-subalgebras in $(\A, \tau)$ and let $\paren{U(t)}_{t\geq0}$ be a unitary free Brownian motion, which is freely independent of $A\vee B$.
	If $a_j \in A$ and $b_j \in B$ for $1\leq j \leq n$, then
	\begin{dmath*}
		\tau\paren{U(\e)a_1U(\e)^*b_1\cdots U(\e)a_nU(\e)^*b_n} 
		= 
		 \frac\e2(\tau\otimes\tau)\left(\delta_{A:B}\left(\sum_{k=1}^n a_kb_k\cdots a_nb_na_1b_1\cdots a_{k-1}b_{k-1}
		- \sum_{k=1}^n b_ka_{k+1}b_{k+1}\cdots a_nb_na_1b_1\cdots b_{k-1}a_k\right)\right)
 +\tau(a_1b_1\cdots a_nb_n) 
		+ \mathcal{O}(\e^2)
	\end{dmath*}
\end{prop}

Important to the proof of the above proposition, and of use to us here also, is the following approximation result.
\begin{prop}[\cite{Voiculescu1999101}*{Proposition 1.4}]
	\label{prop:approxubm}
	Let $A$ be a $W^*$-subalgebra, $(U(t))_t$ a unitary free Brownian motion, and $S$ a $(0,1)$-semicircular element in $(M, \tau)$ so that $A$ and $(U(t))_t$ are $*$-free and $A$ and $S$ are also free. If $a_j \in A$ and $\alpha_j \in \set{1,-1}$, then we have
	$$\tau\paren{\prod_{1\leq j \leq n}^{\rightarrow} a_jU(t)^{\a_j}}
	= \tau\paren{\prod_{1\leq j \leq n}^{\rightarrow} a_j\paren{\paren{1-\frac{t}{2}}+i\a_j\sqrt{t}S}} + \mathcal{O}\paren{t^2},$$
	where the products place the terms in order from left to right.
\end{prop}
Although the proposition was stated in terms of a tracial $W^*$-probability space, traciality was not needed in the proof.

\subsection{A bi-free analogue to the liberation derivation.}
For the remainder of this section, we will always be working in the context of a family of pairs of faces $\paren{(\A_\ell^{(\iota)}, \A_r^{(\iota)})}_{\iota\in\mathcal{I}}$ generating a non-commutative probability space $(\A, \varphi)$.
We will further denote by $\A_\ell$ and $\A_r$ the algebras generated by $\set{\A_\ell^{(\iota)} : \iota \in \mathcal{I}}$ and $\set{\A_r^{(\iota)} : \iota \in \mathcal{I}}$ respectively, and by $\A^{(\iota)}$ the algebra generated by $\A_\ell^{(\iota)}$ and $\A_r^{(\iota)}$.
Moreover, we assume that there are no algebraic relations between $\A^{(i)}$ and $\A^{(j)}$ other than $[\A^{(i)}_\ell, \A^{(j)}_r] = 0$ when $i \neq j$, and possibly $[\A^{(i)}_\ell, \A^{(i)}_r] = 0$.
In particular, we want to ensure that given a product $z_1\cdots z_n$ we can determine the $\chi$-order of the variables.

Suppose $\chi : \set{1, \ldots, n} \to \set{\ell, r}$ and let $1 \leq i, j \leq n$ with $i \preceq_\chi j$.
We denote $[i, j]_\chi := \set{k : i \preceq_\chi k \preceq_\chi j}$ the $\chi$-interval between $i$ and $j$, and define analogously $[i, j)_\chi$, $(i, j]_\chi$, and $(i, j)_\chi$.
Likewise we define $[i, \infty)_\chi := \set{k : 1 \leq k \leq n, i \preceq_\chi k}$ and analogously the other rays.

\begin{defn}
	Fix $\iota \in \mathcal{I}$.
	We define a map
	$$\taur_{\A^{(\iota)} : \bigvee_{j \in \mathcal{I}\setminus\set{\iota}} \A^{(j)}} : \A \to \A \otimes \A$$
	as follows.
	Given $z_i \in \A^{(\epsilon(i))}_{\chi(i)}$,
	$$
		\taur_{\A^{(\iota)} : \bigvee_{j \in \mathcal{I}\setminus\set{\iota}} \A^{(j)}}(z_1\cdots z_n)
		= \sum_{i \in \e^{-1}(\iota)} \sum_{\substack{j\in\e^{-1}(\iota)\\i \preceq_\chi j}}
		z_{[i,j]_\chi^c} \otimes z_{[i,j]_\chi}
		- z_{[i,j)_\chi^c}\otimes z_{[i,j)_\chi}
		- z_{(i,j]_\chi^c}\otimes z_{(i,j]_\chi}
		+ z_{(i,j)_\chi^c}\otimes z_{(i,j)_\chi}.
	$$
	We now extend this definition by linearity to all of $\A$.
	When context makes our intent clear, we will sometimes write $\taur_\iota$ for $\taur_{\A^{(\iota)} : \bigvee_{j \in \mathcal{I}\setminus\set{\iota}} \A^{(j)}}$.

	The subscript $A : B$ is meant to mimic that in the free situation, and the basic properties present there still hold: $B \subset \ker\taur_{A:B}$ and for $a \in A$, $\taur_{A:B}(a) = 1\otimes a-a\otimes 1$.
	However, $\taur_{A:B}$ is not a derivation, even when restricted to the left or right faces of $A$ and $B$.
\end{defn}
\begin{lem}
	\label{lem:taur}
	$\taur_\iota$ is well-defined.
	Moreover, the only terms which do not cancel in the sum defining $\taur_\iota$ are those in which no maximal $\e$-monochromatic $\chi$-interval is split across the tensor sign.
\end{lem}
\begin{proof}
	Our assumptions about the lack of algebraic relations in $\A$ mean that the only ambiguity in writing a product $z_1\cdots z_n$ comes from grouping or failing to group adjacent terms, and commuting left and right terms; the latter has no impact on $\taur_\iota$ because it does not change the $\chi$-ordering of the variables.
	Notice that if $i \prec_\chi i_+$ are consecutive under the $\chi$-ordering and both contribute to the sum, then all intervals with $i$ as an open left endpoint are intervals with $i_+$ as a closed left endpoint and have opposite sign in their contributions to the two terms; likewise, all intervals with $i$ as a closed right endpoint are intervals with $i_+$ as an open right endpoint and again cancel.
	Hence the value of $\taur_\iota$ does not change if a product is written differently, and the only terms which do not cancel are those with the tensor sign falling between two $\e$-monochromatic $\chi$-intervals (or one such interval and the edge of the product), exactly one of which is $\iota$-coloured.
\end{proof}

\begin{rem}
	In essence, $\taur_\iota$ acts by adding one term for each $\chi$-interval with endpoints either before or after terms coming from $\A^{(\iota)}$, consisting of the product of the terms not in that interval tensored with the product of the terms in the interval.
	The sign is chosen so that if the division comes before both chosen nodes or after both chosen nodes the term counts negatively, and otherwise counts positively.
	Notice that when $i = j$, the terms corresponding to $[i,i)_\chi$ and $(i,i)_\chi$ cancel and only one term contributing $-z_1\cdots z_n \otimes 1$ survives.

	The liberation gradient $\delta_{A^{(\iota)}:B}$ can be expressed in a similar manner:
	$$\delta_{A:B}(z_1\cdots z_n) = \sum_{i \in \e^{-1}(\iota)} -z_{(-\infty, i)}\otimes z_{(-\infty, i)^c} + z_{(-\infty, i]} \otimes z_{(-\infty, i]^c}.$$
\end{rem}

\begin{exam}
	Let $\chi, \e$ be as in Example~\ref{ex:vaccine}.
	Then $\taur_{\makeaball{1}}(z_1\cdots z_{10})$ is a sum of the following eight terms:
	
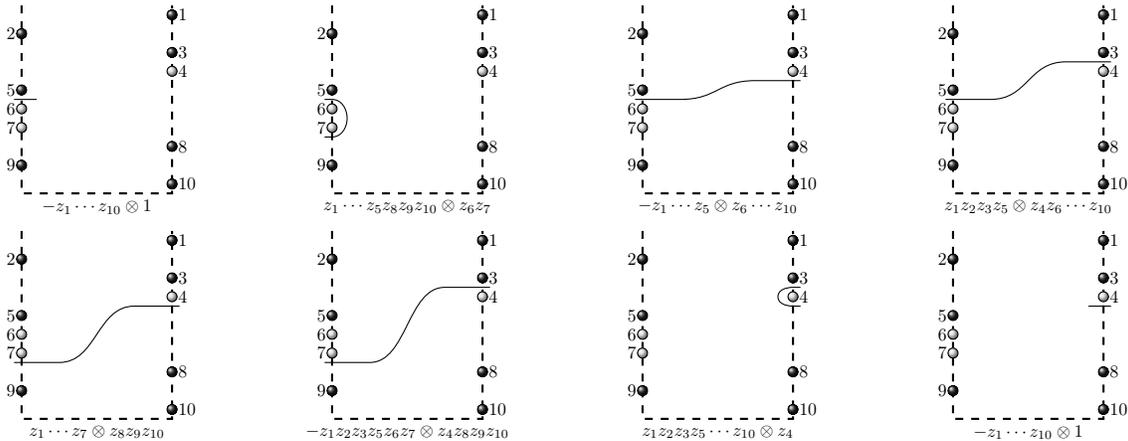
\[\begin{tikzpicture}[baseline]
		\def\colours{{0, 0, 0, 1, 0, 1, 1, 0, 0, 0}}
		\def\sidez{{1,-1,1,1,-1,-1,-1,1,-1,1}}

		\begin{scope}[shift={(-\textwidth*0.375,0)},scale=1]
			\draw[thick, dashed] (-1,0.125) -- (-1, -2.375) --
			node[below,scale=2/3] {$-z_1\cdots z_{10}\otimes 1$}
			(1,-2.375) -- (1,0.125);
			\foreach \y in {0,...,9} {
				\pgfmathtruncatemacro{\nodename}{\y+1}
				\pgfmathtruncatemacro{\sd}{\sidez[\y]}
				\pgfmathparse{\palette[\colours[\y]]}
				\def\clr{\pgfmathresult}
				\node (ball\nodename) [draw, shade, circle, ball color=\clr, inner sep=0.07cm*2/3] at (\sd, -\y*1/4) {};
				\ifthenelse{\sd=1}{\node[scale=2/3,right] at (\sd, -\y*0.25) {\nodename}}
					{\node[scale=2/3,left] at (\sd, -\y*0.25) {\nodename}};
			}
			\draw [rline] (-1.1, -1.125) -- (-.8, -1.125);
		\end{scope}
		\begin{scope}[shift={(-\textwidth*0.125,0)},scale=1]
			\draw[thick, dashed] (-1,0.125) -- (-1, -2.375) --
			node[below,scale=2/3] {$z_1\cdots z_5 z_8z_9z_{10}\otimes z_6z_7$}
			(1,-2.375) -- (1,0.125);
			\foreach \y in {0,...,9} {
				\pgfmathtruncatemacro{\nodename}{\y+1}
				\pgfmathtruncatemacro{\sd}{\sidez[\y]}
				\pgfmathparse{\palette[\colours[\y]]}
				\def\clr{\pgfmathresult}
				\node (ball\nodename) [draw, shade, circle, ball color=\clr, inner sep=0.07cm*2/3] at (\sd, -\y*1/4) {};
				\ifthenelse{\sd=1}{\node[scale=2/3,right] at (\sd, -\y*0.25) {\nodename}}
					{\node[scale=2/3,left] at (\sd, -\y*0.25) {\nodename}};
			}
			\draw [rline] (-1.1, -1.125) -- (-1, -1.125) to [in=90,out=0] (-.8, -1.375) to [in=0,out=270] (-1, -1.625) -- (-1.1, -1.625);
		\end{scope}
		\begin{scope}[shift={(\textwidth*0.125,0)},scale=1]
			\draw[thick, dashed] (-1,0.125) -- (-1, -2.375) --
			node[below,scale=2/3] {$-z_1\cdots z_5\otimes z_6\cdots z_{10}$}
			(1,-2.375) -- (1,0.125);
			\foreach \y in {0,...,9} {
				\pgfmathtruncatemacro{\nodename}{\y+1}
				\pgfmathtruncatemacro{\sd}{\sidez[\y]}
				\pgfmathparse{\palette[\colours[\y]]}
				\def\clr{\pgfmathresult}
				\node (ball\nodename) [draw, shade, circle, ball color=\clr, inner sep=0.07cm*2/3] at (\sd, -\y*1/4) {};
				\ifthenelse{\sd=1}{\node[scale=2/3,right] at (\sd, -\y*0.25) {\nodename}}
					{\node[scale=2/3,left] at (\sd, -\y*0.25) {\nodename}};
			}
			\draw [rline] (-1.1, -1.125) -- ++(0.6,0) .. controls +(.5,0) and +(-.5, 0) .. (0.5, -0.875) -- ++(0.6,0);
		\end{scope}
		\begin{scope}[shift={(\textwidth*0.375,0)},scale=1]
			\draw[thick, dashed] (-1,0.125) -- (-1, -2.375) --
			node[below,scale=2/3] {$z_1z_2z_3z_5\otimes z_4z_6\cdots z_{10}$}
			(1,-2.375) -- (1,0.125);
			\foreach \y in {0,...,9} {
				\pgfmathtruncatemacro{\nodename}{\y+1}
				\pgfmathtruncatemacro{\sd}{\sidez[\y]}
				\pgfmathparse{\palette[\colours[\y]]}
				\def\clr{\pgfmathresult}
				\node (ball\nodename) [draw, shade, circle, ball color=\clr, inner sep=0.07cm*2/3] at (\sd, -\y*1/4) {};
				\ifthenelse{\sd=1}{\node[scale=2/3,right] at (\sd, -\y*0.25) {\nodename}}
					{\node[scale=2/3,left] at (\sd, -\y*0.25) {\nodename}};
			}
			\draw [rline] (-1.1, -1.125) -- ++(0.6,0) .. controls +(.5,0) and +(-.5, 0) .. (0.5, -0.625) -- ++(0.6,0);
		\end{scope}
		\begin{scope}[yshift=-3cm]
		\begin{scope}[shift={(-\textwidth*0.375,0)},scale=1]
			\draw[thick, dashed] (-1,0.125) -- (-1, -2.375) --
			node[below,scale=2/3] {$z_1\cdots z_7\otimes z_8z_9z_{10}$}
			(1,-2.375) -- (1,0.125);
			\foreach \y in {0,...,9} {
				\pgfmathtruncatemacro{\nodename}{\y+1}
				\pgfmathtruncatemacro{\sd}{\sidez[\y]}
				\pgfmathparse{\palette[\colours[\y]]}
				\def\clr{\pgfmathresult}
				\node (ball\nodename) [draw, shade, circle, ball color=\clr, inner sep=0.07cm*2/3] at (\sd, -\y*1/4) {};
				\ifthenelse{\sd=1}{\node[scale=2/3,right] at (\sd, -\y*0.25) {\nodename}}
					{\node[scale=2/3,left] at (\sd, -\y*0.25) {\nodename}};
			}
			\draw [rline] (-1.1, -1.625) -- ++(0.6,0) .. controls +(.5,0) and +(-.5, 0) .. (0.5, -0.875) -- ++(0.6,0);
		\end{scope}
		\begin{scope}[shift={(-\textwidth*0.125,0)},scale=1]
			\draw[thick, dashed] (-1,0.125) -- (-1, -2.375) --
			node[below,scale=2/3] {$-z_1z_2z_3z_5z_6z_7\otimes z_4z_8z_9z_{10}$}
			(1,-2.375) -- (1,0.125);
			\foreach \y in {0,...,9} {
				\pgfmathtruncatemacro{\nodename}{\y+1}
				\pgfmathtruncatemacro{\sd}{\sidez[\y]}
				\pgfmathparse{\palette[\colours[\y]]}
				\def\clr{\pgfmathresult}
				\node (ball\nodename) [draw, shade, circle, ball color=\clr, inner sep=0.07cm*2/3] at (\sd, -\y*1/4) {};
				\ifthenelse{\sd=1}{\node[scale=2/3,right] at (\sd, -\y*0.25) {\nodename}}
					{\node[scale=2/3,left] at (\sd, -\y*0.25) {\nodename}};
			}
			\draw [rline] (-1.1, -1.625) -- ++(0.6,0) .. controls +(.5,0) and +(-.5, 0) .. (0.5, -0.625) -- ++(0.6,0);
		\end{scope}
		\begin{scope}[shift={(\textwidth*0.125,0)},scale=1]
			\draw[thick, dashed] (-1,0.125) -- (-1, -2.375) --
			node[below,scale=2/3] {$z_1z_2z_3z_5\cdots z_{10}\otimes z_4$}
			(1,-2.375) -- (1,0.125);
			\foreach \y in {0,...,9} {
				\pgfmathtruncatemacro{\nodename}{\y+1}
				\pgfmathtruncatemacro{\sd}{\sidez[\y]}
				\pgfmathparse{\palette[\colours[\y]]}
				\def\clr{\pgfmathresult}
				\node (ball\nodename) [draw, shade, circle, ball color=\clr, inner sep=0.07cm*2/3] at (\sd, -\y*1/4) {};
				\ifthenelse{\sd=1}{\node[scale=2/3,right] at (\sd, -\y*0.25) {\nodename}}
					{\node[scale=2/3,left] at (\sd, -\y*0.25) {\nodename}};
			}
			\draw [rline] (1.1, -0.875) -- (1, -0.875) to [in=270,out=180] (.8, -0.75) to [in=180,out=90] (1, -0.625) -- (1.1, -0.625);
		\end{scope}
		\begin{scope}[shift={(\textwidth*0.375,0)},scale=1]
			\draw[thick, dashed] (-1,0.125) -- (-1, -2.375) --
			node[below,scale=2/3] {$-z_1\cdots z_{10}\otimes 1$}
			(1,-2.375) -- (1,0.125);
			\foreach \y in {0,...,9} {
				\pgfmathtruncatemacro{\nodename}{\y+1}
				\pgfmathtruncatemacro{\sd}{\sidez[\y]}
				\pgfmathparse{\palette[\colours[\y]]}
				\def\clr{\pgfmathresult}
				\node (ball\nodename) [draw, shade, circle, ball color=\clr, inner sep=0.07cm*2/3] at (\sd, -\y*1/4) {};
				\ifthenelse{\sd=1}{\node[scale=2/3,right] at (\sd, -\y*0.25) {\nodename}}
					{\node[scale=2/3,left] at (\sd, -\y*0.25) {\nodename}};
			}
			\draw [rline] (1.1, -0.875) -- (0.8, -0.875);
		\end{scope}
		\end{scope}
	\end{tikzpicture}\]
\end{exam}

\begin{thm}
	Let the notation be as above, and suppose $\mathcal{I} = \set{1,2}$.
	Then $(\A_\ell^{(1)}, \A_r^{(1)})$ and $(\A_\ell^{(2)}, \A_r^{(2)})$ are bi-free if and only if $(\varphi\otimes\varphi)\circ\taur_{1} \equiv 0$.
\end{thm}

\begin{proof}
	Suppose first that bi-freeness holds.
	Note that for $\lambda \in \C$, $\taur_1(\lambda) = 0$, so it suffices to check the condition on products $z_1\cdots z_n$ with $z_i \in \A_{\chi(i)}^{(\e(i))}$ and each maximal $\e$-monochromatic $\chi$-interval centred, since an arbitrary term may be written as a sum of such terms.
	However, by Lemma~\ref{lem:taur} we know that each term in $\taur_1(z_1\cdots z_n)$ is a tensor product with zero or more centred $\chi$-intervals occurring on each side of the tensor. The vaccine condition from bi-freeness then tells us that $\varphi\otimes\varphi$ of such a term is $0$, and so $(\varphi\otimes\varphi)\circ\taur_1(z_1\cdots z_n) = 0$.

	Now, suppose that bi-freeness fails, and let $z_1, \ldots, z_n$ be an example of the failure of vaccine with a minimum number of terms.
	Then the only terms which possibly fail to vanish under $\varphi\otimes\varphi$ from $\taur_1$ are those of the form $z_1\cdots z_n\otimes 1$ or $1\otimes z_1\cdots z_n$ (the rest being ones to which vaccine should apply, which are of shorter length and so not counterexamples by minimality).
	The term $z_1\cdots z_n\otimes1$ occurs once per $1$-coloured $\chi$-interval with negative sign, while $1\otimes z_1\cdots z_n$ occurs once with positive sign if the $\chi$-first and $\chi$-last variables are both in $\A^{(1)}$, and not at all otherwise; let $k$ be the number of $1$-coloured $\chi$-intervals, and $d = 1$ if the $\chi$-first and $\chi$-last variables are in $\A^{(1)}$, with $d = 0$ otherwise.
	Hence $\varphi\otimes\varphi(\taur_1(z_1\cdots z_n)) = (d-k)\varphi(z_1\cdots z_n) \neq 0$ unless $k = d$; but if $k = d$ either there is one $1$-interval which is $\set{1, \ldots, n}$, or there are no $1$-intervals, and so $z_1\cdots z_n$ cannot actually be a counterexample of vaccine.
\end{proof}

\subsection{Bi-free unitary Brownian motion.}
We are now ready to introduce a bi-free unitary Brownian motion.

\begin{defn}
	A pair of free stochastic processes $(U_\ell(t), U_r(t))_{t\geq0}$ is a \emph{bi-free unitary Brownian motion} if:
	\begin{itemize}
		\item the multiplicative increments are bi-free: if $0\leq t_1 < \cdots < t_n$, then the family of pairs of faces $((U_\ell^*(t_\iota)U_\ell(t_{\iota+1}), U_r(t_{\iota+1})U_r^*(t_\iota))_{\iota=1}^{n-1}$ is bi-free;
		\item $(U_\ell(t))_{t\geq0}$ and $(U^*_r(t))_{t\geq0}$ are each free unitary Brownian motions, and for all $t > 0$ the $*$-distribution of the pair $(U_\ell(t), U_r(t))$ matches that of $(U_\ell(t), U_\ell^*(t))$; and
		\item the distribution is stationary: the moments of $(U_\ell^*(s)U_\ell(t), U_r(t)U_r^*(s))$ depend only on $t-s$.
	\end{itemize}
\end{defn}

We will show that once again a bi-free unitary Brownian motion may be realized from an additive Brownian motion.
\begin{lem}
	Suppose that $(S_\ell(t))_{t\geq0}$ is a free Brownian motion in a tracial von Neumann algebra $(M, \tau)$, and $J : L^2(M) \to L^2(M)$ is the Tomita operator defined on $M$ by $J(x) = x^*$ and extended continuously to $L^2(M)$.
	Let $S_r(t) = JS_\ell(t)J \in M'$.
	Then if $(U_\ell(t), U_r(t))$ are solutions to the stochastic differential equations
	$$dU_\ell(t) = iU_\ell(t)\,dS_\ell(t) - \frac12 U_\ell(t)\,dt \qquad\text{and}\qquad dU_r(t) = -iU_r(t)\,dS_r(t) - \frac12 U_r(t)\,dt,$$
	with initial conditions $U_\ell(0) = 1 = U_r(0)$, the pair $(U_\ell(t), U_r(t))$ is a bi-free unitary Brownian motion.
	Moreover, $(U_\ell(t), U_r(t))$ converges in distribution as $t \to \infty$ to a Haar pair of unitaries.
\end{lem}

\begin{proof}
	We find immediately that $U_\ell(t)$ is a unitary free Brownian motion.
	Note that integrating a stochastic process $\omega_t\sharp dX_t$ comes down to finding a limit in $L^2(\A)$ of approximations of the form $\sum \theta_{t_k}(x_{t_k}-x_{t_k-1})\phi_{t_k}$, where $\sum \theta_{t_k}\otimes \phi_{t_k}$ approximates $\omega_t$.
	It follows that $d(JX_t^*J) = J(dX_t)^*J$, and in particular, $JdS_r(t)J = dS_\ell(t)$.
	Conjugating the equation for $dU_r(t)$ above, we find
	$$
		d(JU_r(t)J)
		= i\paren{JU_r(t)J}JdS_r(t)J - \frac12\paren{JU_r(t)J}\,dt 
		= i\paren{JU_r(t)J}dS_\ell(t) - \frac12\paren{JU_r(t)J}\,dt.
	$$
	Thus $JU_r(t)J$ satisfies the same differential equation as $U_\ell$, whence the two are equal.
	We conclude that $U_r(t)$ corresponds to right multiplication in the standard representation on $L^2(M)$ by $U_\ell^*(t)$.
	The remaining properties of bi-free unitary Brownian motion now follow readily from the free properties possessed by $(U_\ell(t))_{t\geq0}$; see, e.g., Corollary 10.2.3 of \cite{Charlesworth:2015aa}.
\end{proof}

\begin{rem}
We find that conjugating by bi-free unitary Brownian motion leads to bi-freeness as $t\to\infty$, much like in the free case, and this allows to think of this as a sort of bi-free liberation process.
A strange consequence is the following: suppose that $X, Y \in L^\infty(\Omega, \mu) \subset \A$ are classical random variables, and $((U_\ell(t), U_r(t))_{t\geq 0}$ a bi-free unitary Brownian motion in $\A$, bi-free from $(X, Y)$.
Then $X$ commutes in distribution with $Y$, $U_r(t)$, and $U_r^*(t)$, so in particular, $X$ and $U_r(t)YU_r^*(t)$ become independent as $t\to\infty$ while always generating a commutative probability space.
One finds that
\begin{align*}
	\varphi(f(X)U_r(t)g(Y)U_r^*(t))
	&= \varphi(f(X)g(Y))\varphi(U_r(t))\varphi(U_r^*(t)) + \varphi(f(X))\varphi(g(Y))\paren{1-\varphi(U_r(t))\varphi(U_r^*(t))} \\
	&= \varphi(f(X)g(Y))e^{-t} + \varphi(f(X))\varphi(g(Y))\paren{1-e^{-t}}.
\end{align*}
\end{rem}

We will demonstrate a connection between liberation and the map $\taur$, but first we need a bi-free version of Proposition~\ref{prop:approxubm}.

\begin{lem}
	\label{lem:ubmest}
	Suppose $(\A_\ell, \A_r)$ is a pair of faces in $\A$ and $\paren{U_\ell(t), U_r(t)}$ is a bi-free unitary Brownian motion, bi-free from $(\A_\ell, \A_r)$.
	Suppose further that $(S_\ell, S_r)$ is a pair of semicircular variables with covariance matrix containing a $1$ in every entry, also bi-free from $(\A_\ell, \A_r)$.
	Let $\chi : \set{1, \ldots, n} \to \set{\ell, r}$, and for $1 \leq j \leq n$, take $a_j \in \A_{\chi_j}$ and $\a_j \in \set{1, 0, -1}$.
	Define $\psi : \set{1, \ldots, n} \to \set{1, -1}$ by $\psi(j) = \alpha_j$ if $\chi(j) = \ell$, and $\psi(j) = -\alpha_j$ otherwise.
	Then we have
	$$\varphi\paren{\prod_{1\leq j \leq n}^{\rightarrow} a_jU_{\chi(j)}(t)^{\a_j}}
	= \varphi\paren{\prod_{1\leq j \leq n}^{\rightarrow} a_j\paren{\paren{1-\abs{\a_j}\frac{t}{2}}+i\psi(j)\sqrt{t}S_{\chi(j)}}} + \mathcal{O}\paren{t^2}.$$
\end{lem}

Essentially, this lemma tells us that the pair $(U_\ell(t), U_r(t))$ behaves in $*$-distribution to order $t$ the same as the pair $\paren{1-\frac\e2 + i\sqrt{t}S_\ell, 1-\frac\e2 - i\sqrt{t}S_r}$.

\begin{proof}
	We proceed along the same lines as in the proof of Proposition~\ref{prop:approxubm}.
	Let $I = \set{j : \a_j \neq 0}$, and write $m := \abs{I}$.
	Since the $*$-distribution of $(U_\ell(t), U_r(t))$ is the same as that of $(U_\ell(t), U_\ell^*(t))$, one can check that for any sequence $j_1 < \ldots < j_k$ of terms in $I$,
	$$\varphi\paren{(U_{\chi(j_1)}(t)^{\a_{j_1}}-e^{-t/2})\cdots(U_{\chi(j_k)}(t)^{\a_{j_k}}-e^{-t/2})}
	= -\delta_{k=2}\psi(j_1)\psi(j_2)t + \mathcal{O}(t^2).$$
	This follows from the fact that the same is true in the free case, which was used in the original proof of Proposition~\ref{prop:approxubm} (\emph{cf.} \cite{Voiculescu1999101}).
	
	Now for each $j \in I$, we rewrite $U_{\chi(j)}(t)^{\a_j}$ as $\paren{U_{\chi(j)}(t)^{\a_j} - e^{-t/2}} + e^{-t/2}$, and expand the product on the left hand side of the equation we are trying to establish.
	As we have the estimate $\norm{U_{\chi(j)}(t)^{\a_j} - e^{-t/2}} \leq K\sqrt{t}$, we find that only terms where at most three of these are chosen will contribute more than $\mathcal{O}(t^2)$.
But by the above argument, terms with one or three such differences are $\mathcal{O}(t^2)$ under $\varphi$;
	then only terms which contribute are those where precisely zero or two $\paren{U_{\chi(j)}(t)^{\a_j} - e^{-t/2}}$ terms are chosen.
	Hence,
	\begin{align*}
		\varphi\paren{\prod_{1\leq j \leq n}^{\rightarrow} a_jU_{\chi(j)}(t)^{\a_j}}
		&= \varphi(a_1\cdots a_n)e^{-nt/2} + \mathcal{O}(t^2) \\
		&\qquad- e^{-(n-2)t/2} \paren{\sum_{\substack{1 \preceq_\chi p \prec_\chi q \preceq_\chi n\\p, q \in I}} \varphi\paren{a_1\cdots a_p (U_{\chi(p)}^{\a_p}-e^{-t/2}) a_{p+1}\cdots a_q (U_{\chi(q)}^{\a_q} - e^{-t/2}) a_{q+1}\cdots a_n}} \\
		&= \varphi(a_1\cdots a_n)e^{-nt/2}
		- t e^{-(n-2)t/2} \paren{\sum_{\substack{1 \preceq_\chi p \prec_\chi q \preceq_\chi n\\p, q \in I}} \varphi(a_{(p, q]_\chi})\varphi(a_{(p, q]_\chi^c})\psi(p)\psi(q)}
		+ \mathcal{O}(t^2) \\
		&= \varphi(a_1\cdots a_n)\paren{1 - n\frac{t}2} - t\paren{\sum_{\substack{1 \preceq_\chi p \prec_\chi q \preceq_\chi n\\p, q \in I}} \varphi(a_{(p, q]_\chi})\varphi(a_{(p, q]_\chi^c})\psi(p)\psi(q)} + \mathcal{O}(t^2).
	\end{align*}
	Here the second equality may require some justification.
	One can verify that it is correct by considering the expansion in terms of cumulants; the terms corresponding to partitions with blocks of mixed colour or partitions that do not connect the $U$ terms both vanish, and we are left with all the bi-non crossing partitions which have the two joined. Summing over these, in turn, produces the product of the two moments claimed.

	Next we turn our attention to the right hand side of the equation.
	Notice that the pair $(S_\ell, S_r)$ has the same distribution as $(-S_\ell, -S_r)$ while both are bi-free from $(\A_\ell, \A_r)$, so replacing $\sqrt{t}$ by $-\sqrt{t}$ does not change the value and thus we are in fact working with a power series in $t$ rather than $\sqrt{t}$.
	Since the constant term is clearly correct, we need only establish that the $t$ term agrees.
	Contributions to the linear term come either from selecting a single $\frac{t}{2}$ in the product (together these contribute $-n\frac{t}{2}\varphi(a_1\cdots a_n)$) or from selecting a pair indices to include the semicircular terms from.
	But now
	$$\varphi(a_1\cdots a_{p}(i\psi(p)\sqrt{t})S_{\chi(p)}a_{p+1}\cdots a_{q} (i\psi(q)\sqrt{t})S_{\chi(q)} a_{q+1}\cdots a_n)
	= -t\psi(p)\psi(q)\varphi(a_{(p, q]_\chi})\varphi(a_{(p, q]^c_\chi}).
	$$
	Summing over the terms from which semi-circular elements may be selected, which is to say those with indices coming from $I$, we see the two sides of the claimed equation agree at order $t$, also. 
\end{proof}

\begin{thm}
	Suppose $(\A^{(\iota)}_\ell, \A^{(\iota)}_r)_{\iota \in \set{\makeaball{0}, \makeaball{1}}}$ are algebraically-free pairs of faces in a non-commutative probability space $(\A, \varphi)$, which is bi-free from the bi-free unitary Brownian motion $(U_\ell(t), U_r(t))$.
	Given $\chi : \set{1, \ldots, n} \to \set{\ell, r}$, $\e : \set{1, \ldots, n}\to\set{\makeaball{0}, \makeaball{1}}$, and $x_i \in \A_{\chi(i)}$, set
	$$z_i{(t)} = \left\{\begin{array}{l@{\emph{ if }}l} 
			x_i & \e(i) = \makeaball{0}\\
			U_{\chi(i)}(t) x_i U_{\chi(i)}^*(t) & \e(i) = \makeaball{1}.
	\end{array}\right.$$
	Then we have the following estimate:
	$$
	\varphi(z_1{(t)}\cdots z_n{(t)}) = \varphi(x_1\cdots x_n) + t\varphi\otimes\varphi\paren{\taur_{\makeaball{1}}(x_1\cdots x_n)} + \mathcal{O}(t^2).
	$$
\end{thm}

\begin{proof}
	We first apply Lemma~\ref{lem:ubmest} to replace $U_\ell(t)^{\pm1}$ by $1-\frac{t}{2} \pm i\sqrt{t}S_\ell$ and $U_r(t)^{\pm1}$ by $1-\frac{t}{2}\mp i\sqrt{t}S_r$, for some $(S_\ell, S_r)$ bi-free from $(\A_\ell, \A_r)$ as in Lemma~\ref{lem:ubmest}.
	Again, as the distribution of $(S_\ell, S_r)$ matches that of $(-S_\ell, -S_r)$, we find that we are dealing with a power series in $t$; further, it is evident that the constant term is correct.
	We therefore consider contributions to the linear term.
	
	However, note that these precisely correspond to the terms in the definition of $\taur_{\makeaball{1}}$.
	Indeed, we notice that when $i \prec_\chi j$ with $\e(i) = \e(j) = \makeaball{1}$, selecting the $S$ terms on either side of $x_i$ and $x_j$ contribute a total of
	$$t\paren{\varphi(x_{[i,j]_\chi^c})\varphi(x_{[i,j]_\chi}) - \varphi(x_{[i,j)_\chi^c})\varphi(x_{[i,j)_\chi})
			- \varphi(x_{(i,j]_\chi^c})\varphi(x_{(i,j]_\chi}) + \varphi(x_{(i,j)_\chi^c})\varphi(x_{(i,j)_\chi})}.$$
	The signs occur because the signs of $S$'s $\chi$-before their respective elements, or $\chi$-after, always match.
	This accounts for all the contributions coming from selecting two semicircular variables when expanding the product; what's left are the terms corresponding to selecting a $-\frac{t}{2}$ term, so each $x_i$ coming from \makeaball{1} winds up contributing $-t\varphi(x_1\cdots x_n)$ in total.
	Yet this precisely matches the contribution to $\taur_{\makeaball{1}}$ corresponding to selecting the empty terms with $i = j$.
	We conclude that the linear term in $\varphi(z_1(t)\cdots z_n(t))$ is precisely $t\varphi\otimes\varphi\paren{\taur_{\makeaball{1}}(x_1\cdots x_n)}$.
\end{proof}

\begin{rem}
	In \cite{Voiculescu1999101}, Voiculescu used the free liberation process to define the liberation gradient and a mutual non-microstates free entropy.
	We intend to pursue the bi-free analogue of this approach in a future paper.
\end{rem}


$$\eighthnote$$

\end{document}